\documentclass[preprint,12pt]{elsarticle}

\usepackage[english]{babel}
\usepackage{amsmath, amssymb, amsthm}
\usepackage{amsfonts}

\newtheorem{thm}{Theorem}[section]

\newtheorem{lem}{Lemma}[section]

\newtheorem{prob}{Problem}[section]
\newtheorem{conj}{Conjecture}[section]
\theoremstyle{definition}

\theoremstyle{remark}

\numberwithin{equation}{section}
\numberwithin{table}{section}

\journal{Linear Algebra and Its Applications}

\begin{document}

\begin{frontmatter}



\title{Some results and a conjecture on certain subclasses of graphs according to the relations among certain energies, degrees and conjugate degrees of graphs} 


\author[gazi]{Ercan Alt\i n\i \c{s}\i k\corref{corresponding}}\ead{ealtinisik@gazi.edu.tr}
\author[gazi]{Nur\c{s}ah Mutlu Varl\i o\~{g}lu}\ead{nmutlu@gazi.edu.tr}

\cortext[corresponding]{Corresponding author. Tel.: +90 312 202 1070}
\address[gazi]{Department of Mathematics, Faculty of Sciences, Gazi University \\ 06500 Teknikokullar - Ankara, Turkey}

\begin{abstract}
Let $G$ be a simple graph of order $n$ with degree sequence $(d)=(d_1,d_2,\ldots,d_n)$ and conjugate degree sequence $(d^*)=(d_1^*,d_2^*,\ldots,d_n^*)$. In \cite{AkbariGhorbaniKoolenObudi2010,DasMojallalGutman2017} it was proven that $\mathcal{E}(G)\leq \sum_{i=1}^{n} \sqrt{d_i}$ and $\sum_{i=1}^{n} \sqrt{d_i^*} \leq LEL(G) \leq IE(G) \leq \sum_{i=1}^{n} \sqrt{d_i}$, where $\mathcal{E}(G)$, $LEL(G)$ and $IE(G)$ are the energy, the Laplacian-energy-like invariant and the incidence energy of $G$, respectively, and in \cite{DasMojallalGutman2017} it was concluded that the class of all connected simple graphs of order $n$ can be dividend into four subclasses according to the position of $\mathcal{E}(G)$ in the order relations above. Then, they proposed a problem about characterizing all graphs in each subclass. In this paper, we attack this problem. First, we count the number of graphs of order $n$ in each of four subclasses for every $1\leq n \leq 8$ using a Sage code. Second, we present a conjecture on the ratio of the number of graphs in each subclass to the number of all graphs of order $n$ as $n$ approaches the infinity. Finally, as a first partial solution to the problem, we determine subclasses to which a path, a complete graph and a cycle graph of order $n\geq 1$ belong.   
\end{abstract}

\begin{keyword}
energy \sep Laplacian energy \sep Laplacian-energy-like \sep incidence energy \sep degree sequence \sep conjugate degree sequence.


\MSC 05C50 \sep 05C07

\end{keyword}

\end{frontmatter}


\section{Introduction} \label{}
Let $G$ be a simple graph with vertex set $V(G)=\{v_1,v_2,\ldots , v_n\}$ and edge set $E(G)=\{e_1,e_2,\ldots , e_m\}$. Let $d_i$ be the degree of vertex $v_i$ for $i=1,2,\ldots , n$ such that $d_1 \geq d_2 \geq \cdots \geq d_n.$ The degree sequence of $G$ is the finite sequence $(d)=(d_1,d_2,\ldots , d_n)$ and the conjugate degree sequence of $G$ is the sequence $(d^*)=(d^*_1,d^*_2,\ldots , d^*_n)$, where $d^*=|\{j:d_j\geq i\}|$, see \cite{Molitierno2012}. The adjacency matrix $A(G)$ of $G$ is the matrix whose $ij$-entry is 1 if $v_i$ and $v_j$ are adjacent and 0 otherwise. Its eigenvalues are denoted by $\lambda_1, \lambda_2, \ldots, \lambda_n.$ The energy of the graph $G$ is defined as 
$$
\mathcal{E}(G)= \sum_{i=1}^{n} |\lambda_i|.
$$ 
This concept was defined by I. Gutman \cite{Gutman1978} in 1978. Since then there has been a great interest on this concept and its relatives, see the papers \cite{Gutman2001,Brualdi2006,Nikiforov2016,DasMojallalGutman2017} and the references cited therein. 

The Laplacian of $G$ is $L(G)=D(G)-A(G)$, where $D(G)$ is the diagonal matrix of $d_1, d_2, \ldots , d_n$. It is clear that the Laplacian matrix has nonnegative eigenvalues $\mu_1, \mu_2, \ldots, \mu_n$ such that $\mu_1 \geq \mu_2 \geq \cdots \geq \mu_n=0$. The Laplacian energy $LE(G)$ of $G$ is
$$
LE(G)=\sum_{i=1}^{n} \left| \mu_i-\frac{2m}{n}\right|.
$$  
$LE(G)$ is defined by I. Gutman and B. Zhou \cite{GutmanZhou2006} in 2006. The recent results on this concept can be found in \cite{DasMojallal2015,DasMojallalGutman2016}. 

The Laplacian-energy-like invariant $LEL(G)$ of $G$ is defined as 
$$
LEL(G)=\sum_{i=1}^{n} \sqrt{\mu_i}.
$$
It was defined by Liu and Liu \cite{LiuLiu2008} in 2008. For details on $LEL$ see \cite{DasMojallalGutman2017} and the references cited therein. 

The incidence energy $IE(G)$ of $G$ is the sum of all singular values of the incidence matrix of $G$. One can consult the text \cite{Molitierno2012} for the definition of the incidence matrix of a graph. This type of graph energy was introduced by M. Jooyandeh et al. \cite{JooyandehKianiMirzakhah2009} in 2009. Then, Gutman and et al. \cite{GutmanKianiMirzakhah2009} showed that if $q_1,q_2,\ldots,q_n$ are the eigenvalues of the signless Laplacian matrix $Q(G)$ of $G$, i.e., $ Q(G)=D(G)+A(G)$, then 
$$
IE(G)=\sum_{i=1}^{n} \sqrt{q_i}.
$$
For the recent results, see \cite{DasGutman2014} and the references therein. 

There has been a great interest recently on these graph energies in the literature. In particular, many results which compare certain types of graph energies and graph invariants were published. Indeed, Akbari et al. \cite{AkbariGhorbaniKoolenObudi2010} proved that 
$$
LEL(G) \leq IE(G).
$$ 
Motivated by this result, Das et al. \cite{DasMojallalGutman2017} showed that 
$$
\pi^*(G) \leq LEL(G),\ IE(G) \leq \pi(G) \text{ and }\mathcal{E}(G) \leq \pi(G),
$$
and if $G$ is a threshold graph, then $\pi^*(G)=LEL(G)$, where $\pi^*(G)=\sum_{i=1}^{n} \sqrt{d_i^*}$ and $\pi(G)=\sum_{i=1}^{n} \sqrt{d_i}$. In the same paper, Das et al. \cite{DasMojallalGutman2017} presented the following open problem. 

\begin{prob} \label{problem}[Problem 5.4 \cite{DasMojallalGutman2017}]
\begin{enumerate}
	\item[(1)] Characterize all (connected) graphs for which $\mathcal{E}(G) \leq \pi^*(G)$
	\item[(2)] Characterize all (connected) graphs for which $\pi^*(G) < \mathcal{E}(G) \leq LEL(G)$
	\item[(3)] Characterize all (connected) graphs for which $LEL(G) < \mathcal{E}(G) \leq IE(G)$
	\item[(4)] Characterize all (connected) graphs for which $IE(G) < \mathcal{E}(G) \leq \pi(G)$.
\end{enumerate}
\end{prob}

In the paper of Das et al. \cite{DasMojallalGutman2017} each of left-hand inequalities is originally "less than or equal to." Since $\mathcal{E}(G)$ might be equal to one of these graph invariants or graph energies for some graphs, we have to replace "$\leq$" with "$<$" in left-hand inequalities in (2)-(4). For example, for each odd integer $n\geq3$, $\mathcal{E}(C_n)= IE(C_n)$, where $C_n$ is the cycle graph with $n$ vertices.

For simplicity, we can propose new notations for certain subclasses of graphs investigated in Problem~\ref{problem}. Let $\mathcal{G}_n$ denote the class of all connected simple graphs of order $n$. We also denote the classes of the graphs satisfying the conditions given in (1), (2), (3) and (4) by $\mathcal{G}_n^1$, $\mathcal{G}_n^2$, $\mathcal{G}_n^3$ and $\mathcal{G}_n^4$, respectively. Then we can restate Problem~\ref{problem} as \textit{"Characterize all graphs in each of $\mathcal{G}_n^1$, $\mathcal{G}_n^2$, $\mathcal{G}_n^3$ and $\mathcal{G}_n^4$."}

In this paper, we attack Problem~\ref{problem}. First, using a Sage code \cite{sage}, we find the number of graphs in each class of $\mathcal{G}_n^1$, $\mathcal{G}_n^2$, $\mathcal{G}_n^3$ and $\mathcal{G}_n^4$, separately. Second, in the light of the ratios of the numbers of graphs in these subclasses to the number of graphs in $\mathcal{G}_n $ for each $1\leq n \leq 8$, we present a conjecture about the density of graphs of these subclasses in $\mathcal{G}_n $ as $n$ approaches the infinity. Finally, as a first step in solving Problem~\ref{problem}, we show that a path is in $\mathcal{G}_n^4$, a complete graph is in $\mathcal{G}_n^1$ and a cycle graph is in one of the classes $\mathcal{G}_n^2$, $\mathcal{G}_n^3$ and $\mathcal{G}_n^4$ according to $n=4k$, $n=2k+1$ and $n=4k+2$, respectively, for a positive integer $k$.  

\section{The Number of Graphs in Each Subclasses and A Conjecture on Average Numbers of Graphs}
Sage \cite{sage} has a graph library which includes all connected simple graphs of order $n\leq 8$. Using a Sage code based on the graph library, we count the number of graphs in all subclasses $\mathcal{G}_n^1$, $\mathcal{G}_n^2$, $\mathcal{G}_n^3$ and $\mathcal{G}_n^4$, see Table~2.1. 
\begin{table}[h!]
	\begin{center}
			\label{tab:table1}
		\begin{tabular}{|c|ccccrrrr|}
			\hline
			$n$  & \text{  1  } & \text{  2  } &\text{  3  }& \text{  4  } & \text{  5  } & \text{  6  } & \text{  7  } & \text{ 8 } \\
			\hline
			$|\mathcal{G}_n|$    & 1 & 1  &	2 & 6 & 21 & 112 & 853 & 11117 \\
			$|\mathcal{G}_n^1|$  & 1 & 0  &	0 & 4 & 12 & 58  & 440 & 5586  \\
			$|\mathcal{G}_n^2|$  & 0 & 0  &	0 & 0 & 4 & 39   & 381 & 5463  \\
			$|\mathcal{G}_n^3|$  & 0 & 0  &	1 & 1 & 4 & 12 & 28 & 59 \\
			$|\mathcal{G}_n^4|$  & 0 & 1  &	1 & 1 & 1 & 3  & 4 & 9 \\
			\hline
		\end{tabular}
		\caption{The number of graphs in classes $\mathcal{G}_n, \mathcal{G}_n^1, \mathcal{G}_n^2, \mathcal{G}_n^3$ and $\mathcal{G}_n^4$. }
	\end{center}
\end{table}

Then, it is natural to consider some certain ratios of the numbers of graphs in these four subclasses to the number of graphs in $\mathcal{G}_n$, see Table~2.2. 

\begin{table}[h!]
	\begin{center}
		\label{tab:table2}
		\small
		\begin{tabular}{|c|ccccrrrr|}
			\hline
			$n$  & \text{  1  } & \text{  2  } &\text{  3  }& \text{  4  } & \text{  5  } & \text{  6  } & \text{  7  } & \text{ 8 } \\
			\hline 
			$\frac{|\mathcal{G}_n^1|}{|\mathcal{G}_n|}$  & 1.00000 & 0.00000 & 0.00000 & 0.66667 & 0.57143 & 0.51786  & 0.51583 & 0.50247 \\ 
			$\frac{|\mathcal{G}_n^2|}{|\mathcal{G}_n|}$  & 0.00000 & 0.00000 &	0.00000 & 0.00000 & 0.19048
			 & 034821 & 0.44666  & 0,49141  \\
			$\frac{|\mathcal{G}_n^3|}{|\mathcal{G}_n|}$  & 0.00000 & 0.00000 &	0.50000 & 0.16667 & 0.19048 & 0.10714 & 0.03283 & 0.00531 \\
			$\frac{|\mathcal{G}_n^4|}{|\mathcal{G}_n|}$  & 0.00000 & 1.00000 &	0.50000 & 0.16667 & 0.04762 & 0.02679  & 0.00469 & 0.00081 \\
			\hline
		\end{tabular}
		\caption{Ratios of the numbers of graphs in subclasses to the number of graphs in $\mathcal{G}_n.$}
	\end{center}
\end{table}

In the light of Table~2.2, we can present the following conjecture. 
\begin{conj} \label{conjecture}
Let $n\geq 1$. Then, we have
$$
\lim\limits_{n \rightarrow \infty} \frac{|\mathcal{G}_n^1|}{|\mathcal{G}_n|} = \lim\limits_{n \rightarrow \infty} \frac{|\mathcal{G}_n^2|}{|\mathcal{G}_n|} = \frac{1}{2}
\text{  and  } 
\lim\limits_{n \rightarrow \infty} \frac{|\mathcal{G}_n^3|}{|\mathcal{G}_n|} = \lim\limits_{n \rightarrow \infty} \frac{|\mathcal{G}_n^4|}{|\mathcal{G}_n|} = 0.
$$ 
\end{conj}
Now, we call the limits in Conjecture~\ref{conjecture} \textit{the average numbers of graphs} in four subclasses. Hence, we can restate the conjecture as \textit{"The average numbers of graphs in $\mathcal{G}_n^1$ and $\mathcal{G}_n^2$ is $\frac{1}{2}$, while the average numbers for $\mathcal{G}_n^3$ and $\mathcal{G}_n^4$ is $0$."}   
\section{Some Particular Graphs in The Subclasses}

In this section, we take the first step towards solving Problem~\ref{problem}. To perform this, for all $n\geq 1$, we determine the subclasses to which a path graph, a complete graph and a cycle graph of order $n$ belong. To calculate the energies and related graph invariants of these graphs, we need the following lemma for the sums of values of the sine and the cosine functions at a certain arithmetic progression. For the proof of Lemma~\ref{lemma}, one can consult the book \cite{HermanKuceraSimsa}, see pages 77-78.     
\begin{lem} \cite{HermanKuceraSimsa} \label{lemma}
	\begin{eqnarray*}
		\sum_{j=0}^n \cos(\theta + \alpha j)= \frac{\sin ( \frac{(n+1) \alpha}{2})  \cos\left(\theta+ \frac{n \alpha}{2}\right) }{\sin\left( \frac{\alpha}{2}\right) } \\
	\end{eqnarray*}
	\begin{eqnarray*}
		\sum_{j=0}^n \sin(\theta + \alpha j)= \frac{\sin ( \frac{(n+1) \alpha}{2})  \sin\left(\theta+ \frac{n \alpha}{2}\right) }{\sin\left( \frac{\alpha}{2}\right) } \\
	\end{eqnarray*}
\end{lem}

\begin{thm} 
	Let $P_n$ be a path with $n\geq 2$ vertices. Then, $P_n \in \mathcal{G}_n^4$, that is $IE(P_n) < \mathcal{E}(P_n)  \leq \pi(P_n)$. 
\end{thm}
\begin{proof} 
	
Let $P_n$ be a path with $n$ vertices. Then, its degree sequence is $(d)=(1,2,\ldots,2,1)$ and hence $\pi(P_n)=2+(n-2)\sqrt{2}$. Moreover, the eigenvalues of $Q(P_n)$ are $q_j=2+2\cos(\frac{\pi j}{n})$ $(j=1,2,\ldots,n)$ (see \cite{BrouwerHaemers2011}). Thus, 
\begin{eqnarray*}
IE(P_n)&=&\sum_{j=1}^{n}  \sqrt{2+2\cos\left( \frac{\pi j}{n}\right) } \\
	&=& 2\sum_{j=1}^{n} \cos\left( \frac{ \pi j}{2n}\right) .
\end{eqnarray*}
Then, by Lemma~\ref{lemma} and the sum formula for the sine function, we have
\begin{eqnarray*}
IE(P_n) &=& 2\left(-1+\frac{\sin\left( \frac{(n+1)\pi}{4n}\right)  \cos\left( \frac{\pi}{4}\right) }{\sin\left( \frac{\pi}{4n}\right) }\right) \\
	&=& 2\left(-1+\frac{1}{2} \frac{\sin\left( \frac{\pi}{4n}\right) + \cos\left( \frac{\pi}{4n}\right) }{\sin\left( \frac{\pi}{4n}\right) }\right) \\
	&=& -1+\cot \left( \frac{\pi}{4n}\right)
\end{eqnarray*}
On the other hand, the eigenvalues of the adjacency matrix of $P_n$ are $\lambda_j=2 \cos\left(  \frac{ \pi j}{n+1}\right)  $  $(j=1,2,\ldots,n)$ (see \cite{BrouwerHaemers2011}). Now, we can calculate the energy of $P_n$ in two cases. For the former, suppose $n$ is even.
\begin{eqnarray*}
		\mathcal{E}(P_n)&=& 2\sum_{j=1}^{n}  \left| \cos\left( \frac{ \pi j}{n+1}\right)\right| \\ 
		&=& 4\sum_{j=1}^{\frac{n}{2}}  \cos\left( \frac{ \pi j}{n+1}\right).
\end{eqnarray*}
By Lemma~\ref{lemma} and a product-to-sum formula, we have
\begin{eqnarray*}
\mathcal{E}(P_n) &=& 4\cdot \frac{\sin \left( \frac{(n+2) \pi}{4 (n+1)}\right) \cos \left( \frac{n \pi}{4 (n+1)}\right) }{\sin \left( \frac{\pi}{2 (n+1)}\right) }-4 \\
		&=& 2\cdot \frac{\sin\frac{\pi}{2}+\sin \left( \frac{\pi}{2 (n+1)}\right)}{\sin \left( \frac{\pi}{2 (n+1)}\right)}-4 \\
		&=& -2+2\csc \left( \frac{\pi}{2(n+1)}\right). 
\end{eqnarray*}
For the latter, suppose $n$ is odd. Similarly, by Lemma~\ref{lemma}, we have
\begin{eqnarray*}
		\mathcal{E}(P_n) &=& 4 \sum_{j=1}^{\frac{n-1}{2}} \left| \cos\left( \frac{ \pi j}{n+1}\right)\right| \\ 
		&=& 4 \cdot \frac{\sin\left( \frac{\pi}{4}\right) \cos\left( \frac{(n-1) \pi}{4 (n+1)}\right) }{\sin\left( \frac{\pi}{2 (n+1)}\right)}-4.
\end{eqnarray*}
By the difference formula for the cosine function,
\begin{eqnarray*}
		\mathcal{E}(P_n) &=& 2 \cdot \frac{\cos\left( \frac{\pi}{2 (n+1)}\right)+\sin\left( \frac{\pi}{2 (n+1)}\right) }{\sin\left( \frac{\pi}{2 (n+1)}\right)}-4 \\
		&=& -2+2\cot\left( \frac{\pi}{2 (n+1)}\right).
\end{eqnarray*}
Now, we show that $\mathcal{E}(P_n) > IE (P_n)$. Suppose $n$ is odd. Since the cotangent function is decreasing on the interval $ (0,\pi)$, we have
	\begin{eqnarray*}
		\cot\left( \frac{\pi}{2 n}\right)<\cot\left( \frac{\pi}{2 (n+1)}\right)
	\end{eqnarray*}
for all $n\geq1$. Then,	
\begin{eqnarray*}
		2\cot\left( \frac{\pi}{2 (n+1)}\right)-\cot\left( \frac{\pi}{4 n}\right)&>&	2\cot\left( \frac{\pi}{2 n}\right)-\cot\left( \frac{\pi}{4 n}\right)\\
		&=&	2\cot\left( \frac{\pi}{2 n}\right)-\frac{\cot^2\left( \frac{\pi}{2 n}\right)-1}{2\cot\left( \frac{\pi}{2 n}\right)}\\
		&=&	\frac{4\cot^2\left( \frac{\pi}{2 n}\right)-\cot^2\left( \frac{\pi}{2 n}\right)-1}{2\cot\left( \frac{\pi}{2 n}\right)}\\
		&=&	\frac{3\cot^2\left( \frac{\pi}{2 n}\right)-1}{2\cot\left( \frac{\pi}{2 n}\right)}.
	\end{eqnarray*}
On the other hand, one can show that
	\begin{eqnarray*}
		\frac{3\cot^2\left( \frac{\pi}{2 n}\right)-1}{2\cot\left( \frac{\pi}{2 n}\right)}>1 &\Leftrightarrow& \cot\left( \frac{\pi}{2 n}\right) >3 \\
		&\Leftrightarrow&	n\geq5,
	\end{eqnarray*}
Thus, we obtain $\mathcal{E}(P_n) \geq IE$ for all odd $n\geq5$. Moreover, one can see that $\mathcal{E}(P_n)\geq IE(P_n)$ for $n=1$ and $n=3$, numerically. Thus, $\mathcal{E}(P_n) \geq IE(P_n)$ for all odd $n\geq1$. 
Now, suppose $n$ is even. Since $ \csc x - \cot x > 0 $ for $0<x<\frac{\pi}{2}$, we have $2\csc( \frac{\pi}{2 (n+1)})-\cot\left( \frac{\pi}{4 n}\right) >1$ and hence $\mathcal{E}(P_n)\geq IE(P_n)$ for all even $n\geq5$.
Thus, $\mathcal{E}(P_n)\geq IE(P_n)$ for all $n\geq1$.

Now, we show that $\mathcal{E}(P_n)\leq\pi(P_n)$. Suppose $n$ is odd. Since $\cot x<\frac{1}{x}$ for $0<x<\pi$,
\begin{eqnarray*}
		\mathcal{E} (P_n) = -2+2\cot\left( \frac{\pi}{2 (n+1)}\right)<-2+ \frac{4 (n+1)}{\pi}=\frac{4}{\pi}n+\frac{4}{\pi}-2.
\end{eqnarray*}
On the other hand, one can show that
\begin{eqnarray*}
		\frac{4}{\pi}n+\frac{4}{\pi}-2<\sqrt{2}n+2-2\sqrt{2} \Leftrightarrow n>\frac{4 \pi-4-2\sqrt{2}\pi}{4-\sqrt{2}\pi}\cong0,72.
	\end{eqnarray*}
Thus, $\mathcal{E}(P_n) >\pi(P_n)$ if $n$ is odd. Now, suppose $n$ is even. The Laurent series of the cosecant functions is
\begin{eqnarray*}
		\csc x &=&\sum_{k=0}^{\infty} \frac{(-1)^{k+1}\cdot  2 \cdot (2^{2k-1} -1)B_{2k}}{(2k)!} x^{2k-1}\\
		&=&\frac{1}{x}+\frac{1}{6}x+\frac{7}{360}x^3+\frac{31}{15120}x^5+\cdots,
	\end{eqnarray*}
where $B_{2k}$ is the $2k-$th Bernoulli number. It is well-known that 
$$B_{2k} = \frac{(-1)^{k-1}\cdot 2 \cdot (2k)!} {(2\pi)^{2k}} \zeta(2k),$$
where $\zeta$ is the Riemann zeta function, and $\zeta(2k) \leq \zeta(2)=\pi^2/6$ for every integer $k\geq 1$. Thus,
\begin{eqnarray*}
	-2+2\csc\left( \frac{\pi}{2(n+1)}\right) &<& -2+ \frac{4(n+1)}{\pi}+ \frac{2\pi}{3} \sum_{k=1}^{\infty} \left( \frac{1}{2(n+1)} \right)^{2k-1} \\
	&=& -2+ \frac{4(n+1)}{\pi}+ \frac{4\pi (n+1)}{3(2n+1)(2n+3)} \\ 
	&<& \frac{4(n+1)}{\pi} -1.4  
	\end{eqnarray*}
In what follows one can show that
\begin{eqnarray*}
		\frac{4(n+1)}{\pi} - 1.4 <\sqrt{2}n+2-2\sqrt{2}
\end{eqnarray*}
for all $n\geq5$. Moreover, by a simple calculation, one can easily see that
	
	\begin{eqnarray*}
		-2+2\csc\left( \frac{\pi}{2(n+1)}\right)<\sqrt{2}n+2-2\sqrt{2}
	\end{eqnarray*}
for each $n=2,3,4$. Thus $\mathcal{E}(P_n)<\pi(P_n)$ for all $n\geq2$. The proof is complete.
	
\end{proof}

\begin{thm}
	Let $C_n$ be a cycle graph with $n\geq 3$ vertices.
	\begin{enumerate}
	\item[(i)] If $n=4k$ then $C_n\in \mathcal{G}_n^2$, that is $\pi^*(C_n) < \mathcal{E}(C_n) < LEL(C_n) $,
	\item[(ii)] If $n=2k+1$ then $C_n\in \mathcal{G}_n^3$, that is $LEL(C_n) < \mathcal{E}(C_n) = IE(C_n) $,
	\item[(iii)] If $n=4k+2$ then $C_n\in \mathcal{G}_n^4$, that is $IE(C_n) < \mathcal{E}(C_n) \leq \pi(C_n) $.
	\end{enumerate} 
\end{thm}
\begin{proof}
\textit{\textbf{(i)} }Let $C_n$ be a cycle with $n=4k\geq4$ vertices. Its degree sequence is $(d)=(2,2,\ldots , 2)$ and its conjugate degree sequence is $ (d^*)=(n,n,0,\ldots ,0)$. Thus $\pi^* (C_n)= 2 \sqrt{n}$. On the other hand, the eigenvalues of its adjacency matrix are $2\cos\left( \frac{2 \pi j}{n}\right) $ $(j=0,1,2,\ldots,n-1)$ (see \cite{BrouwerHaemers2011}) and hence its energy is  
\begin{eqnarray*}
		\mathcal{E}(C_n) 
		&=& 2\sum_{j=1}^{4k-1} \left| \cos\left( \frac{ \pi j}{2k}\right) \right|+2 \\ 
		&=& 8\sum_{j=1}^{k-1} \cos\left( \frac{ \pi j}{2k}\right)+4.
\end{eqnarray*} 
By Lemma~\ref{lemma} and the sum formula for the sine function, we have
\begin{eqnarray*}
		\mathcal{E}(C_n)&=& 4\sqrt{2}\cdot \frac{\sin\left( \frac{(k+1) \pi}{4k}\right) }{\sin\left( \frac{\pi}{4k}\right) }-4\\
		&=& 4\cdot \frac{\cos\left( \frac{\pi}{4k}\right) +\sin\left( \frac{\pi}{4k}\right)}{\sin\left( \frac{\pi}{4k}\right) }-4\\
		&=& 4\cot\left( \frac{\pi}{n}\right). 
\end{eqnarray*}
By the fact that $\cot x>\frac{1}{x}+\frac{1}{x-\pi}$ for $ 0<x<\frac{\pi}{2}$, we have
	
	\begin{eqnarray*}
		4\cot\left( \frac{\pi}{n}\right)>\frac{4n (n-2)}{\pi (n-1)}. 
	\end{eqnarray*}
Moreover, one can easily show that 
\begin{eqnarray*}
		\frac{4n(n-2)}{\pi (n-1)}>2\sqrt{n}.
\end{eqnarray*}
Thus, we obtain $\mathcal{E}(C_n)>\pi^*(C_n)$.
	
On the other hand, the Laplacian eigenvalues of $C_n$ are $2-2\cos\left(\frac{2 \pi j}{n} \right) $ $ (j=0,1,\ldots, n-1)$ (see \cite{BrouwerHaemers2011}) and hence its Laplacian energy-like-invariant is
\begin{eqnarray*}
		LEL(C_n)&=& 2\sum_{j=0}^{n-1} \sin\left( \frac{ \pi j}{n}\right)  \\ 
		&=& 2\cot\left( \frac{\pi}{2n}\right).  
\end{eqnarray*}
Here the last step follows from Lemma~\ref{lemma}. Now, let $\theta=\frac{\pi}{2n}$. It is clear that $\cot \theta$ and $\cot 2\theta$ are positive for all $n\geq4$. Then, we have $\frac{\cot^2\theta-1}{\cot \theta}<\cot \theta$ and hence $4\cot 2\theta<2\cot\theta$. Thus, the proof of (i) is complete.

\textit{\textbf{(ii)} } Let $n=2k+1$ for a positive integer $k$. In this case,  
\begin{eqnarray*}
		\mathcal{E}(C_n)&=& 2\sum_{j=0}^{n-1} \left| \cos\left( \frac{2 \pi j}{n}\right) \right|\\
		&=& 2+2\sum_{j=1}^{n-1} \left| \cos\left( \frac{2 \pi j}{n}\right) \right|\\
		&=& -2+4\sum_{j=0}^{\frac{n-1}{2}} \cos\left( \frac{\pi j}{n}\right)\\
		&=& -2+4\cdot \frac{\sin\left( \frac{(n+1) \pi}{4n}\right)\cos\left( \frac{(n-1) \pi}{4n}\right)}{\sin\left( \frac{\pi}{2n}\right)}\\
		&=&2\csc\left( \frac{\pi}{2n}\right)
\end{eqnarray*} 
and then it is clear that $\mathcal{E}(C_n)> LEL(C_n)$. On the other hand, the eigenvalues of the signless Laplacian matrix $Q(C_n)$ are $q_j=2+2\cos\left( \frac{2 \pi j}{n}\right)$ $ (j=0,1,2,\cdots,n-1)$ (see \cite{CvetkovicSimic2009}) and hence the incidence energy of $C_n$ is 
\begin{eqnarray*}
		IE(C_n)&=& -2+4\sum_{j=0}^{\frac{n-1}{2}} \cos\left( \frac{ \pi j}{n}\right)\\
		&=& -2+4\cdot \frac{\sin\left( \frac{(n+1) \pi}{4n}\right)\cos\left( \frac{(n-1) \pi}{4n}\right)}{\sin\left( \frac{\pi}{2n}\right)}\\
		&=&2\csc\left( \frac{\pi}{2n}\right).
	\end{eqnarray*}
In this case, $\mathcal{E}(C_n)=IE(C_n)$. Thus the proof of (ii) is complete.
	
\textit{\textbf{(iii)} } Let $n=4k+2$ for a positive integer $k$. The energy of $C_n$ is  
\begin{eqnarray*}
		\mathcal{E}(C_n)&=& 4\sum_{j=0}^{\frac{n-2}{2}}\left|  \cos\left( \frac{2 \pi j}{n}\right)\right| \\
		&=& 4+8\sum_{j=0}^{\frac{n-1}{8}} \cos\left( \frac{2 \pi j}{n}\right)\\
		&=&-4+\frac{\sin\left( \frac{(n+2) \pi}{4n}\right)\cos\left( \frac{(n-2) \pi}{4n}\right)}{\sin\left( \frac{\pi}{n}\right)}\\
		&=& -4+4\cdot \frac{1+\sin\left( \frac{\pi}{n}\right)}{\sin\left( \frac{\pi}{n}\right)}\\
		&=&4\csc\left( \frac{\pi}{n}\right).
\end{eqnarray*} 
and the incidence energy of $C_n$ is
\begin{eqnarray*}
		IE(C_n)	&=& -2+4\sum_{j=0}^{\frac{n-2}{2}} \left| \cos\left( \frac{\pi j}{n}\right)\right| \\
		&=&-2+2\sqrt{2}\cdot \frac{\cos\left( \frac{(n-2) \pi}{4n}\right)}{\sin\left( \frac{\pi}{2n}\right)}\\
		&=& 2\cot\left( \frac{\pi}{2n}\right).
	\end{eqnarray*}
Since  $\frac{\mathcal{E}(C_n)}{IE(C_n)}=\frac{1}{\cos^2 \left( \frac{\pi}{2n}\right) }>1$, we have $\mathcal{E}(C_n)> IE(C_n)$. The proof of (iii) is complete.
\end{proof}

\begin{thm} 
	Let $K_n$ be a complete graph with $n>3$ vertices. Then $K_n\in \mathcal{G}_n^2$, that is $\mathcal{E}(K_n) \leq \pi^* (K_n)$. 
\end{thm}
\begin{proof}
The eigenvalues of a complete graph of order $n>3$ are 
$$n-1, (-1)^{n-1}, (-1)^{n-1}, \ldots, (-1)^{n-1}, $$
(see \cite{BrouwerHaemers2011}) and hence $\mathcal{E}(K_n) =2n-2$. On the other hand, $d_1=d_2=\cdots=d_n=n-1$ and hence $d_1^*=d_2^*=\cdots=d_{n-1}^*=n$ and $d_n^*=0$. Thus, $\pi^*(K_n)=(n-1)\sqrt{n}$. Since $n>3$, we have $2(n-1)\leq\sqrt{n}(n-1)$, and hence $\mathcal{E}(K_n) \leq\pi^*(K_n).$ 
\end{proof}


\section*{Acknowledgement}
We would like to thank Professor Oktay \"{O}lmez for his valuable help in improving our Sage code. 



\end{document}